\documentclass[a4paper]{amsart}
\usepackage[12pt]{extsizes}

\usepackage[english]{babel}
\usepackage{amsfonts,amsmath,amsthm,amssymb,latexsym,mathrsfs,enumitem}
\usepackage{geometry}
\newtheorem{theorem}{Theorem}[section]

\newtheorem{proposition}[theorem]{Proposition}

\newtheorem{problem}[theorem]{Problem}

\theoremstyle{plain}

\begin{document}

\title[On a problem of Talagrand concerning separately continuous functions]{On a problem of Talagrand concerning separately continuous functions}

\author{Volodymyr Mykhaylyuk}\email{vmykhaylyuk@ukr.net}
\address{Jan Kochanowski University in Kielce, Poland {\it and} Yurii Fedkovych Chernivtsi National University, Ukraine}

\author{Roman Pol}\email{pol@mimuw.edu.pl}
\address{University of Warsaw, Poland}



\begin{abstract}
We construct a separately continuous function $e:E\times K\to\{0,1\}$ on the product of a Baire space $E$ and a compact space $K$ such that no restriction of $e$ to any non-meager Borel set in $E\times K$ is continuous. The function $e$ has no points of joint continuity and hence it provides a negative solution of  Talagrand's problem in \cite{T}.
\end{abstract}

\maketitle

\section{Introduction}

All spaces considered in this note are completely regular and a space $E$ is Baire if and only if any intersection of countably many open sets dense in $E$ is dense. Our terminology follows \cite{GJ}, \cite{Ke} and \cite{K}.

In his seminal paper \cite{N}, Isaac Namioka proved the following theorem (a far-reaching extension of a classical result by Ren\'{e} Baire, cf. \cite{Ke}, 8M): {\it  if $f : E \times K \to \mathbb R$ is a separately continuous function on the product of a Baire space with some additional completeness properties $E$ and a compact space $K$, then there exists a comeager set $G$ in $E$ such that $f$ is jointly continuous at each point of $G \times K$}.

The spaces $E$ for which the assertion of this theorem holds true for any separately continuous function $f : E \times K \to \mathbb R$ with $K$ compact, are called Namioka spaces \cite{MN}. Namioka spaces are Baire \cite{SR} and there are numerous results describing some classes of Baire spaces which are Namioka spaces, cf. \cite{SR}, \cite{D}, \cite{Myk}.

The first example of a Baire space (even a Choquet space, i.e. $\alpha$-favorable space \cite{Ke}) which is not a Namioka space was given by Michael Talagrand \cite{T}, Th\'{e}or\`{e}me 2, and also in this paper the following problem was stated.

\begin{problem}[\cite{T}, Probl\`{e}me 3]
Let $X$ be a Baire space, let $Y$ be a compact space and let $f:X\times Y\to\mathbb R$ be a separately continuous function. Does $f$ have a point of joint continuity?
\end{problem}

The Talagrand problem attracted attention of many mathematicians, cf.  \cite[Problem 3.4]{God}, \cite[Problem 285]{GMZ}, \cite[Problem 7.1]{Mas}.

It was shown in \cite{M} that if $\beta \mathbb N\setminus \mathbb N$ is covered by nowhere dense closed $P$-sets (i.e., sets $A$ such that any $G_\delta$-set containing $A$ is a neighbourhood of $A$), then the Talagrand problem has a negative solution. Since in some models of ZFC this condition is satisfied, cf. \cite{BFM}, the result of Mykhaylyuk \cite{M} provides a negative answer to the Talagrand problem in some models of set theory (however, this condition fails under CH, cf. \cite{KvMM}).

The aim of this note is to show that the approach of Mykhaylyuk, combined with some construction of Kunen, van Mill and Mills, provides a negative solution of the Talagrand problem, without any additional set-theoretic assumptions.

In fact, we shall obtain the following stronger result.

\begin{theorem}\label{th:1.1}
There exists a separately continuous function $e:E\times K\to\{0,1\}$ on the product of a Baire space $E$ and a compact space $K$ such that no restriction of $e$ to any non-meager Borel set in $E\times K$ is continuous.
\end{theorem}

Since the two-valued function $e$ is not continuous on any non-empty open rectangle, $e$ has no points of joint continuity. Also, the function $e$ fails the Baire property, cf. \cite[\S 43]{K}. The space $E$ in Theorem 1.2 is Choquet.

The Namioka theorem triggered also an extensive investigation of the class of compact spaces $K$ such that the assertion of this theorem holds true, whenever $f : E \times K \to \mathbb R$ is separately continuous and $E$ is Baire, i.e. the spaces $K$ with the Namioka property, cf. \cite[VII.7]{DGZ} (in terminology of G. Debs - co-Namioka spaces).

The function $e:E\times K\to\{0,1\}$ we shall construct in the proof of Theorem \ref{th:1.1} gives rise to a function $\varphi:E\times X\to\{0,1\}$, where $X$ is a compact space with the Namioka property, such that $\varphi$ is continuous in the first variable, upper semi-continuous in the second variable, and has no points of joint continuity. This topic, related to the works by Bouziad \cite{Bou}, Debs \cite{Debs} and Mykhaylyuk \cite{M1}, is discussed in Section 4, cf. also Comment 5.4.

Finally, let us point out the following aspect of the topic.

Let $e : E \times K \to \{0,1\}$ be as in Theorem 1.2. Then $e$ induces a map $e^{\ast} : E \to C(K)$ into the algebra of real-valued continuous functions on $K$, defined by $e^{\ast }(f)(x)=e(f,x)$. The separate continuity of $e$ yields continuity of $e^{\ast}$ with respect to the pointwise topology in $C(K)$. In fact, our construction of $e$ guarantees that $e^{\ast}$ is continuous
in the weak topology of the Banach algebra $C(K)$. However, for no non-zero $u \in C(K)$, the multiplication operator $t \to u \cdot e^{\ast}(t)$ has a point of continuity with respect to the norm topology in $C(K)$, cf. \cite{KKM}. This observation is explained in Comment 5.3.

\section{Separately continuous functions without the Baire property}

We shall show in this section that certain extremally disconnected compact spaces $K$ give rise to separately continuous functions $e:E\times K\to\{0,1\}$ described in Theorem \ref{th:1.1}.

In the next section, we shall explain that an example from \cite{KvMM} yields readily a space $K$ which is needed for this approach.

Let $K$ be an extremally disconnected compact space which has a cover $\mathcal P$ by closed nowhere dense $P$-sets (let us recall that this means that for any $L \in \mathcal P$, any countable intersection of neighbourhoods of $L$ is a neighbourhood of $L$)  such that the union of each finite subcollection of $\mathcal P$ is contained in an element of $\mathcal P$, let $E=C(K,\{0,1\})$ be the space of all continuous functions $f:K\to\{0,1\}$ equipped with the topology of uniform convergence on elements of the family $\mathcal P$, i.e., basic neighbourhoods in $E$ of a continuous function $f\in E$ are the sets
$$
N(f,L)=\{g\in C(K,\{0,1\}):g|_L=f|_L\},\,\,L\in\mathcal P,\eqno(2.1)
$$
and let
$$
e:E\times K\to \{0,1\},\,\, e(f,x)=f(x),\eqno(2.2)
$$
be the evaluation map, cf. \cite[Example 3.4]{M}.

Let us recall that Choquet spaces form a very useful class of Baire spaces, cf. \cite{Ke}.

\begin{theorem}\label{th:2.1}
Let $e:E\times K\to \{0,1\}$ be as above. Then $E$ is a Choquet space and the map $e$ is separately continuous but it fails the Baire property on each non-meager Borel set in the product $E\times K$.
\end{theorem}

\begin{proof}
The topology in the function space $E=C(K,\{0,1\})$ being stronger than the pointwise topology, the evaluation map $e$ is separately continuous, cf. (2.1), (2.2).

We shall show that, whenever $E'\times K'$ is a nonempty open rectangle in $E\times K$ and
$$
F_1,\,F_2,\,\dots
$$
are closed nowhere dense sets in $E\times K$, $e$ is not constant on $(E'\times K')\setminus\bigcup\limits_n F_n$. This will show that $E$ is a Baire space and that $e$ restricted to any non-meager Borel set in $E\times K$ fails the Baire property. We shall proceed as follows, cf. \cite[Lemma 2.3]{BP1}, \cite{BP} and \cite{M}.

We shall pick inductively basic neighbourhoods
$$
N(f_1,L_1)\supseteq N(f_2,L_2)\supseteq \dots
$$
in the space $E$, nonempty open-and-closed sets
$$
K=B_0\supseteq B_1\supseteq B_2\supseteq \dots
$$
in the space $K$ and points $x_n,y_n\in L_n\cap B_{n-1}$, $n=1,2,\dots$, such that
$$
N(f_n, L_n)\times B_n \subseteq (E'\times K')\setminus F_n,\eqno(2.3)
$$
$f_n(x_n)=0$ and $f_n(y_n)=1$ for every $n\geq 1$.

Notice that, by (2.1), $L_{1} \subset L_{2} \subset \ldots \, $ and $f_{n+1}$ coincides with $f_{n}$ on $L_{n}$.

Assume that sets $B_0, B_1, \dots , B_{n-1}$, points $x_i, y_i$ for $i\leq n-1$ and $N(f_i,L_i)$  for $i\leq n-1$ are already defined. At the $n$-th stage, we choose first
$$N(f,L)\times B_n \subseteq (E'\times K')\setminus F_n$$ where $L\in\mathcal P$, $B_n\subseteq B_{n-1}$, and, if $n>1$, also $N(f,L)\subseteq N(f_{n-1},L_{n-1})$ (in particular, $L_{n-1}\subseteq L$).

Since $L$ is closed and nowhere dense,  one can find distinct points $x_n,y_n\in B_{n-1}\setminus L$ and next, one can pick $L_n\in\mathcal P$ such that
$$L\cup \{x_n,y_n\}\subseteq L_n.$$
Then, we choose $f_n\in E$ so that $f_n$ coincides with $f$ on $L$, $f_n(x_n)=0$ and $f_n(y_n)=1$.

Let
$$
A_{n,d}=\{x\in L_n:f_n(x)=d\},\,\, A_d=\bigcup\limits_n A_{n,d},
$$
for $d=0,1$.

The set $K\setminus A_1$ is a $G_\delta$-set containing $A_0$, and $A_{0,n}$ being compact $P$-sets, there are open-and-closed sets $U_{n.0}$ in $K$ such that
$$
A_{n,0}\subseteq U_{n.0}\subseteq K\setminus A_1.
$$
Let $U_0=\bigcup\limits_n U_{n,0}$. Then $K\setminus U_0$ is a $G_\delta$-set in $K$ containing $A_1$, and similarly, there is an open $\sigma$-compact set $U_1$ in $K$ containing $A_1$ and disjoint from $U_0$.

Since $K$ is extremally disconnected, the sets $U_0$ and $U_1$ have disjoint closures in $K$, and therefore, $\overline{A_0}\cap\overline{A_1}=\emptyset$. Any function in $E$ which is zero on $\overline{A_{0}}$ and one on $\overline{A_{1}}$ belongs to the intersection $\bigcap\limits_n N(f_n, L_n)$.

Now let $f\in\bigcap\limits_n N(f_n, L_n)$ and let $x_0$ and $y_0$ be cluster points of sequences $(x_n)$ and $(y_n)$  respectively. Since $x_n,y_n\in B_{n-1}$ for every $n$, $x_0,y_0\in \bigcap\limits_n B_n$. Moreover, $f(x_n)=0$ and $f(y_n)=1$ for every $n$. Therefore, $f(x_0)=0$ and $f(y_0)=1$.
Now, we have, cf. (2.3),
$$(f,x_0), (f,y_0)\in\bigcap\limits_n(N(f_n, L_n)\times B_n)\subseteq 
(E'\times K')\setminus\bigcup\limits_n F_n,$$
$$
e(f,x_0)=f(x_0)=0\quad{\rm and}\quad e(f,y_0)=f(y_0)=1.
$$
Moreover, in the course of the proof, choosing the neighbourhoods $N(f_n,L_n)$ and showing that $$\bigcap\limits_n N(f_n, L_n)\ne\emptyset,$$ we have established also that $E$ is a Choquet space. This completes the proof.
\end{proof}

\section{Proof of Theorem \ref{th:1.1}}

\subsection{The space $X$ of Kunen, van Mill and Mills}

A key element of the construction of a compact space $K$ with the required properties will be the following space $X$ from Example 1.2 in \cite{KvMM}.

The space $X$ is the set of non-decreasing functions $f:\omega_2\to\omega_1+1$, considered as the subspace of the Tychonoff product of $\omega_2$ copies of the space of all ordinals $\leq \omega_1$ endowed with the order topology.

As was pointed out in \cite[Section 3.1]{KvMM}, for every $\alpha<\omega_2$ the set
$$
A_\alpha=\{x\in X:x(\alpha)=\omega_1\}$$
and for every $\xi<\omega_1$ the set
$$
A^\xi=\{x\in X:x(\beta)\leq\xi\,\,{\rm for}\,\,{\rm all}\,\,\beta<\omega_2\}
$$
are nowhere dense closed $P$-sets in $X$, and the collection $\mathcal E$ of these sets covers $X$.

Moreover, the families $\{ A_\alpha:\alpha<\omega_2 \}$ and $\{ A^\xi:\xi<\omega_1 \}$ are increasing. Therefore, for each countable subfamily $\mathcal A$ of $\mathcal E$ there are $\alpha<\omega_2$ and $\xi<\omega_1$ such that
$$
\bigcup\mathcal A \subseteq A_\alpha\cup A^{\xi}.
$$

\subsection{The projective cover $K$ of the space $X$}
Gleason's results \cite{G} (cf. \cite{R}) provide an extremally disconnected compact space $K$ and a continuous irreducible surjection $\pi:K\to X$ onto the Kunen, van Mill and Mills space $X$, considered in subsection 3.1. Let us adopt the notation introduced in this subsection.

Let $\mathcal P$ be the collection of finite unions of elements $\pi^{-1}(A)$, where $A\in\mathcal E$. Then, $\pi$ being irreducible, the collection $\mathcal P$ in the extremally disconnected compact space $K$ has the properties stated at the beginning of section 2, and in effect, Theorem \ref{th:2.1} provides a justification of Theorem \ref{th:1.1}.

\section{An example concerning the Namioka property}

The space $X$ described in Section 3.1 has the Namioka property. We shall check this below, establishing a stronger property that the function space $C(X)$ has a $\tau_p$-Kadec renorming, i.e., there exists a norm $\| \cdot \|$ on $C(X)$, equivalent to the supremum norm, such that the pointwise topology $\tau_p$ and the norm topology coincide on the unit sphere $\{ u \in C(X): \| u \|=1\}$.

The fact that an existence of a $\tau_p$-Kadec norm implies the Namioka property is well-known, but a bit hidden in the literature we are aware of. Therefore, let us briefly explain the situation. Deville and Godefroy \cite{DG} proved that if $C(X)$ has a $\tau_p$-Kadec norm $\| \cdot \|$ which is $\tau_p$-lsc (i.e., the function $u \to \| u \|$ is lsc with respect to $\tau_p$) the $X$ has the Namioka property. More specifically, this follows readily from the proof of Lemma IV-1 in \cite{DG}, cf. \cite{God}, a remark preceding Problem 3.4. Now, as was pointed out by Raja \cite[Proposition 4]{Raj},  all $\tau_p$-Kadec norms on $C(X)$ are $\tau_p$-lsc, cf. also \cite[Proposition 2.2]{BKT}.

The mapping $e:E\times K\to\{0,1\}$ constructed in Sections 2 and 3 gives rise to the following result, related to the works by Bouziad \cite{Bou}, Debs \cite{Debs} and Mykhaylyuk \cite{M1}.

\begin{proposition}\label{pr:1} There is a function $\phi:E\times X\to\{0,1\}$ on the product of a Choquet space $E$ and a compact space $X$ with the Namioka property such that $\phi$ is continuous in the first variable and upper-semicontinuous in the second variable, but no restriction of $\phi$ to a non-meager Borel set in $E\times X$ is continuous.
In particular, $\phi$ has no points of joint continuity and it fails the Baire property.
\end{proposition}

\begin{proof} (A). The Choquet space $E$ and the compact space $X$ are the spaces considered in Sections 2 and 3. Let us recall that $E=C(K,\{0,1\})$ is the set of continuous function $f:K\to\{0,1\}$ on a compact space $K$ which maps onto $X$ by a continuous irreducible surjection $\pi:K\to X$, and basic neighbourhoods $N(f,L)$ of points $f\in E$ are defined by $(2.1)$.

Let $e:E\times K\to\{0,1\}$ be the separately continuous map defined by $(2.2)$. We shall define $\phi:E\times X\to\{0,1\}$ by the formula
$$
\phi(f,x)=\left\{\begin{array}{ll}
                       0, & {\rm if} \,\, f \,\,{\rm is} \,\, {\rm zero} \,\, {\rm on} \,\,\pi^{-1}(x),\\
                       1, & {\rm otherwise}.
                       \end{array}
 \right.
$$
Let us fix $x\in X$, and let $A \in \mathcal E$ contain $x$, cf. 3.1. Then, by 3.2, $L =\pi ^{-1}(A) \in \mathcal P$ and $\pi^{-1}(x)\subset L$. If $f\in E$, then  $\phi$ is constant on $N(f,L)\times\{x\}$, cf. 2.1, i.e. the map
$$
f\rightarrow \phi(f,x)
$$
is locally constant, hence $\phi$ is continuous in the first variable. On the other hand, if we fix $f\in E=C(K,\{0,1\})$, the set
$$
 \{x\in X:\phi(f,x)=0\}=\{x\in X: \pi^{-1}(x)\subseteq f^{-1}(0)\}
$$
is open in $X$ and hence $\phi$ is upper-semicontinuous in the second variable (let us notice that, $\pi$ being irreducible, the set $\{x\in X:\phi(f,x)=0\}$ is an open domain, cf. \cite{K}, and therefore, the map
$$
x\rightarrow \phi(f,x)
$$
is also quasi-continuous, cf. \cite{M1}).

$(B)$. To check that no restriction of $\phi$ to any non-meager Borel set in $E\times X$ is continuous, we shall slightly modify the reasoning in the proof of Theorem \ref{th:2.1}.

Let $E'\times X'$ by any nonempty rectangle in $E\times X$ and let
$$
H_1, H_2, \dots
$$
be closed nowhere dense sets in $E\times X$. We have to check that $\phi$ takes both values on $(E'\times X')\setminus \bigcup\limits_i H_i$.

Let $K'=\pi^{-1}(X')$ and $F_i=(id\times \pi)^{-1}(H_i)$. Since $\pi$ is irreducible, the sets $F_i$ are closed and nowhere dense in $E\times K$.

Let us adopt the notations of the proof of Theorem \ref{th:2.1}. As in this proof,  we shall choose inductively basic neighbourhoods $N(f_n,L_n)$ in $E$ and open-and-closed sets $B_n$ in $K$, introducing the following change: we shall not need the points $x_n$ and $y_n$, but we demand that $B_n=\pi^{-1}(V_n)$, where $V_n$ is open-and-closed in $X$ (since $\pi$ is irreducible, such a choice of $B_n$ is always possible).

We shall also modify the final part of the proof of Theorem \ref{th:2.1}. Having defined the sets
$$
L_1\subseteq L_2\subseteq \dots
$$
we appeal to the property of the space $X$ indicated at the end of Section 3.1, to pick $L\in\mathcal P$ such that $\bigcup\limits_n L_n\subseteq L$. Since $L$ is closed and nowhere dense, there are points $c_n\in B_n\setminus L$ and let
$$
C=\{c_1,c_2,\dots\}.
$$
Then $K\setminus C$ is a $G_\delta$-set containing the $P$-set $L$ and hence, there is an open neighbourhood $U$ of $L$ disjoint from $C$. The compact sets
$$
B_1\setminus U\supseteq B_2\setminus U\supseteq \dots
$$
are nonempty, and so is the set  $\bigcap\limits_n B_n\setminus U$, disjoint from $L$. Since both $\bigcap\limits_n B_n = \pi^{-1}(\bigcap\limits_n V_n)$ and $L$ are full preimages under $\pi$, there is $x\in X$ such that
$$
\pi^{-1}(x)\subseteq \bigcap\limits_n B_n\setminus L.
$$
Now, the reasoning in Section 2 provides two continuous functions $f,g\in C(K,\{0,1\})$  such that $g|_{L_n}=f_n|_{L_n}=h|_{L_n}$, $g$ is zero on $\pi^{-1}(x)$
and $h$ is one on $\pi^{-1}(x)$. In effect,
$$
(g,x), (h,x) \in (E'\times X')\setminus \bigcup\limits_n H_n,
$$
but $\phi(g,x)=0$ and $\phi(h,x)=1$.

$(C)$. To complete the proof, we have to make sure that the Kunen, van Mill, Mills space $X$ has the Namioka property. In fact, using the result from \cite{BKT}, we shall show that $X$ has the stronger property: the Banach algebra $C(X)$ of real-valued functions on $X$ has a $\tau_p$-Kadec norm.

For every $\alpha < \omega_2$ and every $\xi\leq \omega_1$ we put
$$
X^\alpha_\xi=\{x\in X:x(\gamma)=\xi\,\,{\rm for}\,\,{\rm all}\,\,\alpha\leq\gamma<\omega_2\}.
$$
Moreover, for each $\beta< \alpha$ let $p^{\alpha,\beta}_\xi:X^\alpha_\xi\to X^\beta_\xi$ be the natural retraction. Using a transfinite induction on $\alpha< \omega_2$ we shall show that all spaces $C(X^\alpha_\xi)$, $0\leq\xi\leq\omega_1$, have a $\tau_p$-Kadec renorming.

Since all spaces $X^0_\xi$ are singletons,  all spaces $C(X^0_\xi)$ have a $\tau_p$-Kadec renorming. Assume that for some $\alpha < \omega_2$ all spaces $C(X^\beta_\xi)$, $\beta<\alpha$ and $\xi\leq\omega_1$, have a $\tau_p$-Kadec renorming. Let $\alpha$ be a limit ordinal. Applying to the inverse sequence $\{X^\gamma_\xi; p^{\beta,\gamma}_\xi:0\leq\gamma<\beta\leq\alpha\}$ \cite[Lemma 4.7]{BKT}, we infer that $C(X^\alpha_\xi)$ has a $\tau_p$-Kadec renorming for every $\xi\leq\omega_1$.

Now, let $\alpha=\beta+1$. We fix $\xi\leq\omega_1$. For every $\eta\leq \xi$ we put
$$
Y_\eta=\{x\in X^\alpha_\xi:x(\beta)=\eta\} \,\,\,{\rm and}\,\,\,Z_\eta=\{x\in X^\alpha_\xi:x(\beta)\leq\eta\}.
$$
Notice that every $Y_\eta$ is homeomorphic to $X^\beta_\eta$ and consequently, every space $C(Y_\eta)$ has a $\tau_p$-Kadec renorming. Using a transfinite induction on $\eta\leq \xi$ we shall show also that every space $C(Z_\eta)$ has a $\tau_p$-Kadec renorming.

It is obvious that the space $C(Z_0)$ has a $\tau_p$-Kadec renorming. Assume that for some $\eta\leq \xi$ all spaces $C(Z_\zeta)$, $\zeta<\eta$, have $\tau_p$-Kadec renormings. If $\eta=\zeta+1$ then $C(Z_\eta)$ has a $\tau_p$-Kadec renorming, because $Z_\eta$ is the direct sum of $Z_\zeta$ and $Y_\eta$.

Now let $\eta$ be a limit ordinal and let, for each $\tau<\zeta\leq\eta$,  $\,q_{\zeta,\tau}:Z_\zeta\to Z_\tau$ be the retraction
$$
q_{\zeta,\tau}(x)(\gamma)=\left\{\begin{array}{ll}
                        \min \{x(\gamma),\tau\}, & \gamma\leq \beta;\\
                         \xi, &\gamma\geq\alpha.
                       \end{array}
 \right.
$$
The inverse sequence $\{Z_\tau; q_{\zeta,\tau}:0\leq\tau<\zeta\leq\eta\}$ satisfies the assumptions of \cite[Lemma 4.7]{BKT} and therefore, $C(Z_\eta)$ has a $\tau_p$-Kadec renorming. Hence, since $Z_\xi=X^\alpha_\xi$, $C(X^\alpha_\xi)$ has a $\tau_p$-Kadec renorming.

In effect, for every $\alpha<\omega_2$ and $\xi\leq \omega_1$ the space $C(X^\alpha_\xi)$  has a $\tau_p$-Kadec renorming, and so does $X$, as $X$ is the limit of the inverse system of the spaces  $X^{\alpha}_{\omega_1}$, $\alpha<\omega_2$, cf. \cite[Lemma 4.7]{BKT}.
\end{proof}

\section{Comments}
\subsection{Some cardinality issues} In our proof of Theorem \ref{th:1.1}, the weight of each of the spaces $E$ and $K$ is $2^{\aleph_2}$. In some models of ZFC, one can have $2^{\aleph_2}=2^{\aleph_0}$, cf. \cite[Theorem IV.7.17]{Ku}. As we pointed out, the Baire space $E$ is Choquet, and by \cite[Theorem 2.2]{M}, the existence of a separately continuous everywhere discontinuous function on the product $B\times L$ of a Choquet space $B$ and a compact space $L$, implies that the weight of both $B$ and $L$ is at least $2^{\aleph_0}$.

Our approach does not provide in ZFC such spaces with the minimal possible weight. Let us notice, however, that there are (in ZFC) separately continuous maps $f:B\times L\to \{0,1\}$, with $B$ Choquet and $L$ compact, both of weight $2^{\aleph_0}$, which are not Borel measurable, cf. \cite{BP}.

\subsection{Concerning Theorem \ref{th:2.1}}
In this theorem,  it is enough to assume that that $K$ is an $F$-space, cf. \cite{GJ}. The reasoning justifying Theorem \ref{th:2.1} requires in this case only minor modifications.

Let us also notice that the property of Kunen, van Mill, Mills space stated at the end of Section 3.1 yields that the Choquet space $E$ we construct in the proof of Theorem \ref{th:1.1} is also a $P$-space, i.e. all $G_\delta$ sets in $E$ are open.

\subsection{The induced map $e^* : E \to C(K)$} The evaluation map $e:E\times K\to\{0,1\}$ in Theorem \ref{th:2.1} induces a map $e^* : E \to C(K)$ into the Banach algebra of real-valued continuous functions on $K$, defined by $e^*(f)(x)=e(f,x)$. Since $e$ is separately continuous, $e^*$ is continuous with respect to the pointwise topology in $C(K)$. However, in fact $e^*$ is continuous with respect to the weak topology in the Banach space $C(K)$. To that end, it is enough to make sure that for each Radon measure $\mu$ on $K$, the support of $\mu$ is contained in some element of the collection $\mathcal P$.

We will show first that for every $L\in\mathcal P$ there is an open neighbourhood $U(L)$ of $L$ such that $\mu (U(L) \setminus L) = 0$. Let $L\in\mathcal P$. Then there is a $\sigma$-compact set $F$ disjoint from $L$ with $\mu (K\setminus L) = \mu (F)$. Since $L$ is a $P$-set, there is an open neighbourhood $U(L)$ of $L$ which is disjoint from $F$. Now, we have $\mu (U(L) \setminus L) = 0$.

There are $L_1, ...L_n \in \mathcal P$ such that $U(L_1), ..., U(L_n)$ cover $K$, and in effect, the support of $\mu$ is contained in the union of $L_1, ...L_n$, which is contained in an element of $\mathcal P$.

Let $u \in C(K)$ be a non-zero function, and let $M_{u}: C(K) \to C(K)$ be the multiplication operator $M_{u}(f)=u\cdot f$.

We shall check that the composition $M_{u} \circ e^{\ast}: E \to C(K)$ has no points of continuity with respect to the supremum norm in $C(K)$. To this end let us fix $\delta >0$ and a nonempty open set $W$ in $K$ such that $|u(x)|\geq \delta$ for $x \in W$.

Let us pick any nonempty open set $U$ in $E$. Since $e$ is non-constant on $U \times W$, there is $(s,a) \in U \times W$ with $e(s,a)=0$ and let $W' \subset W$ be an open neighbourhood of $a$ in $K$ such that $e$ is zero on $\{s\} \times W'$. Then, $e$ being non-constant on $U \times W'$, we get $(t,b) \in U \times W'$ with $e(t,b)=1$. In effect, $b \in W$ and $e^{\ast}(s)(b)=0$, $e^{\ast}(t)(b)=1$. Therefore $\| M_{u} \circ e^{\ast}(s) - M_{u} \circ e^{\ast}(t)\| \geq \delta$, i.e., the norm-oscillation of $M_{u} \circ e^{\ast}$ on $U$ is at least $\delta$.

\subsection{Concerning Proposition \ref{pr:1}}

One can show that if a Baire space $B$ is $\beta$-defavorable in the game $\mathcal J'(B)$ defined by Debs \cite{Debs}, then any function $f:B\times Z\to\{0,1\}$ on the product of $B$ and a compact space $Z$, continuous in the first variable and upper-semicontinuous in the second variable, has points of joint continuity.

The function $\phi$ constructed in the proof of Proposition \ref{pr:1} is also quasi-continuous in the second variable. We refer the reader to \cite{Bou} and \cite{M1} for some positive results concerning points of joint continuity of functions of two variables, continuous in one variable and quasi-continuous in the other one.

\end{document}